\documentclass[12pt]{amsart}
\usepackage{amssymb,latexsym}
\usepackage{enumerate}
\usepackage{hyperref}

\makeatletter \@namedef{subjclassname@2010}{%
  \textup{2010} Mathematics Subject Classification}
\makeatother

\newtheorem{Theorem}{Theorem}

\newtheorem{Lemma}{Lemma}
\newtheorem{Corollary}{Corollary}
 \newtheorem*{UnnumberedTheorem}{Theorem}
 \newtheorem*{Problem}{Problem}

\newcommand{\FF}[0]{\mathbb F}

\renewcommand{\aa}[0]{\textbf{\textit{a}}}

\renewcommand{\bar}[1]{\overline{#1}}

\newcommand{\eps}[0]{\varepsilon}

\newcommand{\leg}[2]{\left(\frac{#1}{#2}\right)}

\newcommand{\lr}[1]{\left(#1\right)}

\begin{document}


\baselineskip=17pt



\title{Estimates for Character Sums with Various Convolutions}

\author[Brandon Hanson]{Brandon Hanson} \address{Pennsylvania State University\\
University Park, PA}
\email{bwh5339@psu.edu}

\date{}
\maketitle

\begin{abstract}
We provide estimates for sums of the form \[\left|\sum_{a\in A}\sum_{b\in B}\sum_{c\in
C}\chi(a+b+c)\right|\]
and 
\[\left|\sum_{a\in A}\sum_{b\in B}\sum_{c\in
C}\sum_{d\in D}\chi(a+b+cd)\right|\]
when  $A,B,C,D\subset \FF_p$, the field with $p$ elements and $\chi$ is a non-trivial multiplicative character modulo $p$.
\end{abstract}
\section{Introduction}
\sloppy
In analytic number theory, one is often concerned with estimating a bilinear sum
of the form \begin{equation}\label{bilinear}S=\sum_{\substack{1\leq m\leq M\\
1\leq n\leq N}}a_mb_nc_{m,n}\end{equation} where $a_m,\ b_n$ and $c_{m,n}$ are
complex numbers. The standard way to handle this sum is to apply the
Cauchy-Schwarz inequality so that \begin{align*}|S|^2&\leq \lr{\sum_{1\leq m\leq
M}|a_m|\left|\sum_{1\leq n\leq N}b_nc_{m,n}\right|}^2\\&\leq\lr{\sum_{1\leq
m\leq M}|a_m|^2}\lr{\sum_{1\leq n_1,n_2\leq N}b_{n_1}\bar{b_{n_2}}\sum_{1\leq m\leq
M}c_{m,n_1}\bar{c_{m,n_2}}}.\end{align*} \noindent One usually has that
\[\sum_{1\leq m\leq M}c_{m,n_1}\bar{c_{m,n_2}}\] is small when $n_1\neq n_2$, so
that the second factor is essentially dominated by the \emph{diagonal terms}
where $n_1=n_2$.

For instance, suppose $p$ is a prime number and denote by $\FF_p$ the field with $p$ elements. We
write $e_p(u)=e^{2\pi i u/p}$ and we denote by $\chi$ a multiplicative (or Dirichlet) character
modulo $p$.
Two well-known sums of the form (\ref{bilinear}) are
\begin{equation}\label{Paley}S_\chi(A,B)=\sum_{a\in A}\sum_{b\in
B}\chi(a+b)\end{equation} and
\begin{equation}\label{exponential}T_x(A,B)=\sum_{a\in A}\sum_{b\in
B}e_p(xab)\end{equation}  where $A$ and $B$ are subsets of $\FF_p$.

By the triangle inequality, each of these sums are at most $|A||B|$, but we expect an upper bound of the form $|A||B|p^{-\eps}$ for some positive $\eps$. Indeed, using the Cauchy-Schwarz inequality as above, and
orthogonality of characters, one can prove that the sums (\ref{Paley}) and
(\ref{exponential}) are at most $(p|A||B|)^{1/2}$. Such an estimate is better
than the trivial estimate when $|A||B|>p$.

For the second sum, (\ref{exponential}), the bound $(p|A||B|)^{1/2}$ is quite
sharp. Indeed, if $A=B=\{n:1\leq n\leq \delta p^{1/2}\}$ for a small number
$\delta>0$, then products $ab$ with $a,b\in A$ are at most $\delta^2p$ (here we are identifying residues mod $p$ with integers between $0$ and $p-1$). It follows that $|e_p(ab)-1|\ll \delta^2$, so the summands in (\ref{exponential}) are
essentially constant and there is little cancellation. On the other
hand, it is conjectured that the first sum, (\ref{Paley}), should exhibit
cancellation even for small sets $A$ and $B$. From now on, we will call
(\ref{Paley}) the \emph{Paley sum}. The problem of obtaining good estimates for
it beyond the range $|A||B|>p$ appears to be quite hard.

In this article we investigate character sums which are related to the Paley
sum. First, we motivate its study with the following
question of S\'ark\"ozy:

\begin{Problem}[S\'ark\"ozy]\label{Sarkozy}
Are the quadratic residues modulo $p$ a sumset? That is, do there exist sets 
$A, B\subset \FF_p$ each of size at least two, and with the set $A+B$ equal 
to the set of quadratic residues?
\end{Problem}

One expects that the answer to the above question is no.
Heuristically, if $B$ contains two elements $b$ and $b'$ we would require that
$A+b$ and $A+b'$ are both subsets of the quadratic residues. But we expect that
$a+b$ is a quadratic residue half of the time, and we expect that $a+b'$ also be
a residue half of the time \emph{independent} of whether or not $a+b$ is a
quadratic residue. So if $A+B$ consisted entirely of quadratic residues then
many unlikely events must have occurred. For $A+B$ to consist of all the
quadratic residues would be shocking. The difficulty in this problem
is establishing the aforementioned independence.

In \cite{Sh}, Shkredov showed the the quadratic residues are never of the form
$A+A$. In more recent work, \cite{Sh2}, he also ruled out the case that $Q=A+B$ when $A$ is a multiplicative subgroup. By way of character sum estimates, Shparlinski, building on work of S\'ark\"ozy \cite{Sar} has proved that:

\begin{UnnumberedTheorem}[S\'ark\"ozy, Shparlinski]
If $A, B\subset \FF_p$, each of size at least two with the set $A+B$ equal to 
the set of quadratic residues then $|A|$ and $|B|$ are within a constant factor of $\sqrt p$.
\end{UnnumberedTheorem}
\noindent As a consequence of this theorem and a combinatorial theorem of Ruzsa,
one can deduce that the quadratic residues are not of the for $A+B+C$ with each
set of size at least two.

S\'ark\"ozy's question is settled by improved bounds for
the Paley sum. Since each sum $a+b$ with $a\in A$ and $b\in B$ is a quadratic residue we
have \[|A||B|=\sum_{a\in A}\sum_{b\in B}\leg{a+b}{p}\leq (p|A||B|)^{1/2}.\] So
$|A||B|\leq p$ and this estimate just fails to resolve S\'ark\"ozy's problem.
So even improving upon the bound
$S_{\leg{\cdot}{p}}(A,B)\leq (p|A||B|)^{1/2}$ by a constant factor would be
worthwhile.

Breaking past this barrier, often called the \emph{square-root
barrier}, is hard. In practice, the usual way we estimate character sums is via
the method of completion. One way of doing so was outlined at the beginning
of this article. With this method, we replace a short sum over a subset
$A\subset \FF_p$ with a complete sum over the whole of $\FF_p$ which lets us to use orthogonality. However some terms, the diagonal terms, exhibit no
cancellation at all and must be accounted for. By completing the sum we create
more diagonal terms, and the resulting loss becomes worse than trivial when the set $A$ is too small. One can dampen the loss from completion by using a higher moment (using H\"older's inequality as opposed to Cauchy-Schwarz). This was the
idea used by Burgess in his work on character sums in \cite{Bu1} and \cite{Bu2}, and it is still one of the only manoeuvres we have for
pushing past the square-root barrier. Still, with higher moments the
off-diagonal terms become more complicated and we must settle for worse
orthogonality estimates, which can be limiting.

In the case of the Paley sum, the square-root barrier is more than just a
consequence of our methods. Suppose $q=p^2$ so that $\FF_p$ is a subfield of
$\FF_q$ and each element in $\FF_p$ is the square of an element in $\FF_q$.
Since $\FF_p$ is closed under addition, any sum $a+b$ with $a,b\in\FF_p$ is also
a square in $\FF_q$. So, if we take $A=B=\FF_p$ and $\chi$ the quadratic
character on $\FF_q$, then there is no cancellation in $S_\chi(A,B)$. This
shows that, for the Paley sum over $\FF_q$, the bound $|S_\chi(A,B)|\leq
(q|A||B|)^{1/2}$ is essentially best possible. In order to improve the bound for the Paley
sum past the square-root barrier, we need to use an argument which is sensitive
to the fact that $\FF_p$ has no subfields. Such arguments are hard to come by
and this is perhaps the greatest source of difficulty in the problem.

There have been improvements to estimates for the Paley sum when the sets $A$
and $B$ have a particularly nice structure. In \cite{FI}, Friedlander and
Iwaniec improved the range in which one can obtain non-trivial estimates when
the set $A$ is an interval. This constraint was weakened by Mei-Chu Chang in \cite{C1}
to the case where $|A+A|$ is very small:

\begin{UnnumberedTheorem}[Chang]\label{Chang}
Suppose $A,B\subset \FF_p$ with $|A|,|B|\geq p^\alpha$ for some
$\alpha>\frac{4}{9}$ and such that $|A+A|\leq K|A|$. Then there is a constant
$\tau=\tau(K,\alpha)$ such that for $p$ sufficiently large and any non-trivial
character $\chi$, we have \[|S_\chi(A,B)|\leq
|A||B|p^{-\tau}.\]
\end{UnnumberedTheorem} 
\noindent We remark that in light of Freiman's Theorem, which we will recall
shortly, the condition that $|A+A|$ has to be so small is still very
restrictive.

Often problems involving a sum of two variables,
called \emph{binary additive problems}, are hard. Introducing a third
variable gives rise to a \emph{ternary additive problem}, which may be 
tractable. In this paper we establish non-trivial bounds beyond the square-root
barrier for character sums with more than two variables. These results are
different from those mentioned above since they hold for all sets which are sufficiently large
- there are no further assumptions made about their structure. Our first theorem
is the following.

\begin{Theorem}\label{TripleSum}
Given subsets $A,B,C\subset \FF_p$ each of size $|A|,|B|,|C|\geq \delta\sqrt p$,
for some $\delta>0$, and a non-trivial character $\chi$, then we have
\[\left|\sum_{a\in A}\sum_{b\in B}\sum_{c\in
C}\chi(a+b+c)\right|=o_\delta(|A||B||C|).\]
\end{Theorem}

There are analogous results for exponential sums. We mentioned above that the
sum $T_x(A,B)$ in (\ref{exponential}) also obeys the bound $|T_x(A,B)|\leq (p|A||B|)^{1/2}$.
While this bound may be sharp, Bourgain \cite{Bou} proved that with more
variables one can extend the range in which the estimate is non-trivial.

\begin{UnnumberedTheorem}[Bourgain]
There is a constant $C$ such that the following holds. Suppose $\delta>0$ and
$k\geq C\delta^{-1}$, then for $A_1,\ldots,A_k\subset \FF_p$ with $|A_i|\geq
p^\delta$ and $x\in\FF_p^\times$, we have \[\left|\sum_{a_1\in
A_i}\cdots\sum_{a_k\in A_k}e_p(xa_1\cdots
a_k)\right|<|A_1|\cdots|A_k|p^{-\tau}\] where $\tau>C^{-k}$.
\end{UnnumberedTheorem}

We cannot prove results of this strength. The reason is that one can play the
additive and multiplicative structures of the frequencies appearing in such
exponential sums and then leverage the Sum-Product Phenomenon to deduce some
cancellation. The structure of multiplicative characters is not so nice and we
rely on Burgess' method instead.

In Theorem \ref{TripleSum}, we would prefer a bound of the form
$|S_\chi(A,B,C)|\leq |A||B||C|p^{-\tau}$ for some positive $\tau$. However,
the proof of Theorem \ref{TripleSum} relies on Chang's Theorem, which only allows one to
estimate $S_\chi(A,B)$ past the square-root barrier under the hypothesis that
$|A+A|\leq K|A|$ for some constant $K$. This hypothesis plays a crucial part in
the proof of her theorem because it allows for the use of Freiman's
Classification Theorem:

\begin{UnnumberedTheorem}[Freiman]
Suppose $A$ is a finite set of integers such that $|A+A|\leq K|A|$. Then there
is a generalized arithmetic progression $P$ containing
$A$ and such that $P$ is of dimension at most $K$ and $\log(|P|/|A|)\ll K^{c}$
for some absolute constant $c$.
\end{UnnumberedTheorem}

Using this classification theorem, one can make a change of variables $a\mapsto
a+bc$, which is the first step in a Burgess type argument. Freiman's Theorem is unable to accommodate the situation $|A+A|\leq |A|^{1+\delta}$, even for
small values of $\delta>0$, which is what is needed in order
to get a power saving in our bound for ternary character sums. To circumvent the
use of Freiman's Theorem, we can replace triple sums with sums of four
variables. By incorporating both additive and multiplicative convolutions we
arrive at sums of the form \[H_\chi(A,B,C,D)=\sum_{a\in A}\sum_{b\in
B}\sum_{c\in C}\sum_{d\in D}\chi(a+b+cd).\] In this way we have essentially
\emph{forced} a scenario where we can make use of the Burgess argument. By
introducing both arithmetic operations, we are able to weigh the additive
structure in one of the variables against the multiplicative structure of that
variable in order to use a Sum-Product estimate. Our second result is:

\begin{Theorem}\label{MixedSum}
Suppose $A,B,C,D\subset \FF_p$ are sets with $|A|,|B|,|C|,|D|>p^\delta$,
$|C|<\sqrt p$ and $|D|^4|A|^{56}|B|^{28}|C|^{33}\geq p^{60+\eps}$ for some
$\delta, \eps>0$. There is a constant $\tau>0$ depending only on $\delta$ and $\epsilon$ such that
\[|H_\chi(A,B,C,D)|\ll |A||B||C||D|p^{-\tau}.\] In the case that $|A|,|B|,|D|>p^\delta$,
$|C|\geq \sqrt p$ and $|D|^8|A|^{112}|B|^{56}\geq p^{87+\eps}$ then there is a
constant $\tau>0$ depending only on $\delta$ and $\epsilon$ such that \[|H_\chi(A,B,C,D)|\ll |A||B||C||D|p^{-\tau}.\]
\end{Theorem}

Theorem \ref{MixedSum} is simplified greatly when all sets in question are assumed to have roughly the same size:

\begin{Corollary}\label{MixedSum2}
Suppose $A,B,C,D\subset \FF_p$ are sets with $|A|,|B|,|C|,|D|>p^\delta$ with $\delta>\frac{1}{2}-\frac{1}{176}$. Then $H_\chi(A,B,C,D)\leq |A||B||C||D|p^{-\eps}$ for some $\eps>0$ depending only on $\delta$.
\end{Corollary}

\section*{Acknowledgements}
The author is grateful to John Friedlander and Antal Balog for much fruitful
discussion during the preparation of this article.

\section{Background}
Here we recall facts concerning multiplicative characters over finite
fields and additive combinatorics. For details concerning character sums, we
refer to Chapters 11 and 12 of \cite{IK}. The reference \cite{TV} is extremely
helpful for all things additive combinatorial.

Multiplicative characters are the characters $\chi$ of the group $\FF_q^\times$
which are extended to $\FF_q$ by setting $\chi(0)=0$. In order to carry out the
proof of a Burgess-type estimate, we shall need Weil's bound for character sums
with polynomial arguments. 

\begin{Theorem}[Weil]
Let $f\in\FF_p[x]$ be a polynomial with $r$ distinct roots over $\bar{\FF_p}$.
Then if $\chi$ has order $l$ and provided $f$ is not an $l$'th power over
$\bar{\FF_p}[x]$ we have
\[\left|\sum_{x\in\FF_p}\chi(f(x))\right|\leq r\sqrt p.\]
\end{Theorem}

\begin{Lemma}\label{MomentBound}
Let $k$ be a positive integer and $\chi$ a non-trivial multiplicative character.
Then for any subset $A\subset\FF_p$ we have \[\sum_{x\in\FF_q}\left|\sum_{a\in
A}\chi(a+x)\right|^{2k}\leq |A|^{2k}2k\sqrt p+(2k|A|)^kp.\]
\end{Lemma}
\begin{proof}
Expanding the $2k$'th power and using that
$\bar\chi(y)=\chi(y^{p-2})$, we have
\begin{align*}
&\sum_{a_1,\ldots,a_{2k}\in
A}\sum_x\chi((x-a_1)\cdots(x-a_k)(x-a_{k+1})^{p-2}\cdots(x-a_{2k})^{p-2})\\
&=\sum_{\aa\in
A^{2k}}\sum_x\chi(f_{\aa}(x)).
\end{align*}
Here $f_{\aa}(t)$ is the
polynomial
\[f_{\aa}(X)=(X-a_1)\cdots(X-a_k)(X-a_{k+1})^{p-2}\cdots(X-a_{2k})^{p-2}.\]
By Weil's theorem, $\sum_x\chi(f_{\aa}(x))\leq 2k\sqrt p$ unless $f_{\aa}$ is an
$l$'th power, where $l$ is the order of $\chi$. If any of the roots $a_i$ of
$f_{\aa}$ is distinct from all other $a_j$ then it occurs in the above
expression with multiplicity 1 or $p-2$. Both $1$ and $p-2$ are prime to $l$
since $l$ divides $p-1$. Hence $f_{\aa}$ is an $l$'th power only provided all of its roots can be grouped into pairs. So, for all but at most
$\frac{(2k)!}{2^k k!}\leq (2k|A|)^k$ vectors $\aa\in A^{2k}$, we have the
estimate $2k\sqrt p$ for the inner sum. For the remaining $\aa$ we bound the sum
trivially by $p$. Hence the upper bound \[\sum_{x\in\FF_q}\left|\sum_{a\in
A}\chi(a+x)\right|^{2k}\leq |A|^{2k}2k\sqrt p+(2k|A|)^kp.\]
\end{proof}

We now turn to results from additive combinatorics. Let $A$ and $B$ be finite
subsets of an abelian group $G$. The \emph{additive energy} between $A$ and $B$
is the quantity \[E_+(A,B)=\left|\{(a,a',b,b')\in A\times A\times B\times B:a+b=a'+b'\}\right|.\]
One of the fundamental results on additive energy is the
Balog-Szemer\'edi-Gowers Theorem, which we use in the following form.

\begin{Theorem}[Balog-Szemer\'edi-Gowers]\label{BSG}
Suppose $A$ is a finite subset of an abelian group $G$ and \[E_+(A,A)\geq
\frac{|A|^3}{K}.\] Then there is a subset
$A'\subset A$ of size $|A'|\gg\frac{|A|}{K(\log(e|A|))^2}$ with \[|A'-A'|\ll
K^4\frac{|A'|^3(\log (|A|))^8}{|A|^2}.\] The implied constants are absolute.
\end{Theorem}

This version of the Balog-Szemer\'edi-Gowers Theorem has very good explicit
bounds, and is due Bourgain and Garaev. The proof is essentially a combination
of the Lemmas 2.2 and 2.4 from \cite{BG}. It was communicated to us by O.
Roche-Newton. Since we prefer to work with sumsets rather than difference sets
we have the following lemma which is a well-known application of Ruzsa's
Triangle Inequality.

\begin{Lemma}\label{sumset}
Suppose $A$ is a finite subset of an abelian group $G$. Then \[|A-A|\leq
\lr{\frac{|A+A|}{|A|}}^2|A|.\]
\end{Lemma}

We will prefer to work with the energy between a set and itself
rather than between distinct sets, so we need the following fact, which is a simple consequence of the Cauchy-Schwarz inequality.
\begin{Lemma}\label{triangle}
For sets $A$ and $B$ we have \[E_+(A,B)^2\leq E_+(A,A)E_+(B,B)\]
\end{Lemma}

We now record a general version of Burgess' argument, which is an
application of H\"older's inequality and Weil's bound. This proof is distilled
from the proof of Burgess's estimate in Chapter 12 of \cite{IK}.

\begin{Lemma}
\label{BasicBurgess}
Let $A,B,C\subset \FF_p$ and suppose $\chi$ is a non-trivial multiplicative
character. Define \[r(x)=|\{(a,b)\in A\times B:ab=x\}|.\] Then for any
positive integer $k$, we have the estimate
\begin{align*}
\sum_{x\in\FF_p}r(x)\left|\sum_{c\in
C}\chi(x+c)\right|&\leq(|A||B|)^{1-1/k}E_\times(A,A)^{1/4k}E_\times(B,B)^{1/4k}\cdot\\
&\cdot\lr{|C|^{2k}2k\sqrt p+(2k|C|)^kp}^{1/2k}.
\end{align*}
\end{Lemma}
\begin{proof}
Call the left hand side above $S$. Applying H\"older's
inequality \begin{align*}
|S|&\leq\lr{\sum_{x\in\FF_p}r(x)}^{1-1/k}\lr{\sum_{x\in\FF_p}r(x)^2}^{1/2k}\lr{\sum_{x\in\FF_p}\left|\sum_{c\in
C}\chi(x+c)\right|^{2k}}^{1/2k}\\
&=T_1^{1-1/k}T_2^{1/2k}T_3^{1/2k}.
\end{align*}
Now $T_1$ is precisely $|A||B|$ and $T_2$ is the multiplicative energy
$E_\times(A,B)$. By Lemma \ref{triangle} inequality, we have
\[E_\times(A,B)\leq\sqrt{E_\times(A,A)E_\times(B,B)}.\]
The estimate for $T_3$ is an immediate from Lemma \ref{MomentBound}.
\end{proof}

The last ingredient in our proof is the most crucial. Sum-Product estimates are
sensitive to prime fields and allow us to break the square-root barrier. We
record the following estimate of Rudnev. 

\begin{Theorem}[Rudnev]
\label{EnergyEstimate}
Let $A\subset \FF_p$ satisfy $|A|<\sqrt p$. Then \[E_\times(A,A)\ll
|A||A+A|^\frac{7}{4}\log |A|.\]
\end{Theorem}

This is not the state of the art for Sum-Product theory in $\FF_p$, which at the time of this writing is found in
\cite{RNRS}, but the above estimate is more readily applied to our situation. Moreover, the strength of the Sum-Product estimates is not the bottleneck for proving non-trivial character sum estimates in a wider range (avoiding completion is).

\section{Ternary sums}\label{Triple}

We begin this section by giving a simple estimate which is non-trivial past the
square-root barrier provided we can control certain additive energy.

\begin{Lemma}\label{energy}
Given subsets $A,B,C\subset \FF_p$ and a non-trivial character $\chi$ we have
\[|S_\chi(A,B,C)|\leq\sqrt{p|A|E_+(B,C)}.\]
\end{Lemma}
\begin{proof}
Let $r(x)$ be the number of ways in which $x\in\FF_p$ is a sum $x=b+c$ with
$b\in B$ and $c\in C$. Then
\begin{align*}|S(A,B,C)|&\leq\sum_{x\in\FF_p}r(x)\left|\sum_{a\in
A}\chi(a+x)\right|\\ &\leq
\lr{\sum_{x\in\FF_p}r(x)^2}^{1/2}\lr{\sum_{x\in\FF_p}\left|\sum_{a\in
A}\chi(a+x)\right|^2}^{1/2}.\end{align*} It is straightforward to check that the
first factor above is $\lr{E_+(B,C)}^{1/2}$ and as before, the second factor is
$\lr{p|A|}^{1/2}$.
\end{proof}

\begin{Lemma}\label{arg}
Let $z_1,\ldots,z_n$ be complex numbers with $|\arg z_1-\arg z_j|\leq \delta$.
Then \[|z_1+\ldots +z_n|\geq (1-\delta)(|z_1|+\ldots+|z_n|).\]
\end{Lemma}
\begin{proof}
We have
\begin{align*}|z_1|+\ldots+|z_n|&=\theta_1z_1+\ldots+\theta_nz_n\\&=\theta_1(z_1+\ldots+z_n)+(\theta_2-\theta_1)z_2+\ldots+(\theta_n-\theta_1)z_n\end{align*}
for some complex numbers $\theta_k$ of modulus 1 with $|\theta_1-\theta_j|\leq
\delta$.
Thus by the triangle inequality \[|z_1|+\ldots+|z_n|\leq
|z_1+\ldots+z_n|+\delta(|z_2|+\ldots+|z_n|)\] and the result follows.
\end{proof}

We are now able prove Theorem \ref{TripleSum}. Ignoring technical details for the moment, either we are in a situation where Lemma \ref{energy} improves upon the trivial
estimate, or else we can appeal to the Balog-Szemer\'edi-Gowers Theorem and
deduce that $A$ has a subset with small sumset. In the latter case we can make
use of Chang's Theorem and also arrive at a non-trivial estimate, even saving a
power of $p$.
Unfortunately, this second scenario does not come in to play until one of the
sets has a lot of additive energy. This means that the saving from Lemma \ref{energy}
will become quite poor before we are rescued by Chang's estimate. We proceed
with the proof proper.

\begin{proof}[Proof of Theorem \ref{TripleSum}] Suppose, by way of
contradiction, that the theorem does not hold. This means that there is some
positive constant $\eps>0$ such that for $p$ arbitrarily large, we have sets
$A,B,C\subset\FF_p$ with $|A|,|B|,|C|\geq \delta\sqrt{p}$, and a non-trivial
character $\chi$ of $\FF_p^\times$ satisfying
\[|S_\chi(A,B,C)|\geq\eps|A||B||C|.\] It follows that
\[\eps|A||B||C|\leq\sum_{a\in A}|S_\chi(B,a+C)|.\] If we let \[A'=\{a\in
A:|S_\chi(B,a+C)|\geq \frac{\eps}{2}|B||C|\}\] then
\[\frac{\eps}{2}|A||B||C|\leq \sum_{a\in A'}|S_\chi(B,a+C)|\] and $|A'|\geq
|A|\eps/2$. Now by the same argument as in the proof of Lemma \ref{energy}, we
must have \[\frac{\eps^2}{4}|A|^2|B|^2|C|^2\leq p|C|E_+(A',B)\leq
p|C|E_+(A',A')^{1/2}E_+(B,B)^{1/2},\] the last inequality being a consequence of
Lemma \ref{triangle}. So, using that $|A|,|B|,|C|\geq \delta\sqrt p$ and
$E_+(B,B)\leq |B|^3$, we have \[E_+(A',A')\geq \frac{\eps^4\delta^4}{16}|A'|^3\]
and so by Theorem \ref{BSG} and Lemma \ref{sumset} we can find a subset
$A''\subset A'$, with size at least $(\eps\delta)^{t}\sqrt p$ and such that
$|A''+A''|\leq (\eps\delta)^{-t}|A''|$ for some $t=O(1)$. Now since $A''\subset
A'$, we have \[\frac{\eps}{2}|A''||B||C|\leq \sum_{a\in A''}|S_\chi(B,a+C)|.\]
By the pigeon-hole principle, after passing to a subset of $A'''$ of size
$|A''|/16$, we can assume that the complex numbers $S_\chi(B,a+C)$ all have
argument within
 $\frac{1}{2}$ of each other. Thus, by Lemma \ref{arg}, we have
 \[\frac{\eps}{4}|A'''||B||C|\leq \left|S_\chi(A''',B,C)\right|,\] we have
 $|A'''|\geq (\eps\delta)^{t}\sqrt
p/16$, and we have \[|A'''+A'''|\leq|A''+A''|\leq (\eps\delta)^{-t}|A''|\leq
16(\eps\delta)^{-t}|A'''|.\] However, by the triangle inequality, this implies
that \[\frac{\eps}{4}|A'''||B+c|\leq \max_{c\in
C}\left|S_\chi(A''',B+c)\right|.\] This is in clear violation of Theorem
\ref{Chang} provided $p$ is sufficiently large in terms of $\delta$ and $\eps$.
Thus we have arrived at the desired contradiction.
\end{proof}

\section{Mixed quaternary sums}\label{Mixed}

We now turn to the estimation of the sums $H_\chi(A,B,C,D)$. 
First we consider an auxiliary ternary character sum with a
multiplicative convolution. \[M_\chi(A,B,C)=\sum_{a\in A}\sum_{b\in B}\sum_{c\in
C}\chi(a+bc).\] We can bound $M_\chi$ in terms of the \emph{multiplicative energy}
\[E_\times(X,Y)=|\{(x_1,x_2,y_1,y_2)\in X\times X\times Y\times
Y:x_1y_1=x_2y_2\}|.\] As before, this satisfies the bound
\[E_\times(X,Y)^2\leq E_\times(X,X)E_\times(Y,Y).\]

Now, using Sum-Product estimates, if the sets had enough additive structure, we could
bound the multiplicative energy non-trivially and make an improvement.
This is essentially Burgess' argument, though he did not use Sum-Product theory;
rather, since he was working with arithmetic progressions, the multiplicative
energy could be bounded directly. 

By fixing one element in the sum $H_\chi(A,B,C,D)$, we can view
it as a ternary sum in two different ways. First,
\[H_\chi(A,B,C,D)=\sum_{d\in D}S_\chi(A,B,d\cdot C)\] where $d\cdot C$ is the
dilate of $C$ by $d$. We can use Lemma \ref{energy} to bound this sum 
non-trivially whenever we can bound $E_+(C,C)$ non-trivially. If not, we can write
\[H_\chi(A,B,C,D)=\sum_{a\in A}M_\chi(a+B,C,D)\] instead and try to bound this
non-trivially using Lemma \ref{BasicBurgess}, which we can do if $E_\times(C,C)$
is smaller than $|C|^3$. By making some simple manipulations to $H_\chi$ and using a
sum-product estimate, we will be able to guarantee one of these facts holds.

Before presenting our proof, we mention that A. Balog has communicated to us a forthcoming result with T. Wooley which asserts:

\begin{UnnumberedTheorem}
There is a positive $\delta$ such that any $X\subset \FF_p$ can be decomposed as $X=Y\cup Z$ with $E_+(Y,Y)\leq |Y|^{3-\delta}$ and $E_\times(Z,Z)\leq |Z|^{3-\delta}$.
\end{UnnumberedTheorem}

The proof of this result uses ideas similar to those in our proof of Theorem \ref{MixedSum}, and implies a non-trivial estimate for $H_\chi$. Indeed, decomposing $C=Y\cup Z$ as in the theorem,
\[|H_\chi(A,B,C,D)|\leq |H_\chi(A,B,X,D)|+|H_\chi(A,B,Y,D)|.\] Estimating each of these sums as we mentioned above, gives a non-trivial bound for $|H_\chi(A,B,C,D)|$.

Now we proceed to our proof of Theorem \ref{MixedSum}.

\begin{proof}[Proof of Theorem \ref{MixedSum}]
Let $2\leq k \ll \log p$ be a (large) parameter. First we handle the
case $|C|<\sqrt p$. Let us write \[|H_\chi(A,B,C,D)|=\Delta|A||B|||C||D|\] so
that our purpose is to estimate $\Delta$. Let \[C_1=\left\{c\in C:|S_\chi(A,B,c\cdot
D)|\geq\frac{\Delta|A||B||D|}{2}\right\}.\] We have that for any $C_2\subset C_1$
\[\frac{|C_2|}{2|C|}|H_\chi(A,B,C,D)|= |C_2|\frac{\Delta|A||B||D|}{2}\leq \sum_{c\in C_2}|S_\chi(A,B,c\cdot D)|,\] and
using that the inner quantities are at most $|A||B||D|$, we also have \[|C_1|\geq
\frac{\Delta}{2}|C|.\] Now, passing to a subset $C_2$ of $C_1$ of size at
least \[|C_2|\geq \frac{|C_1|}{16}\geq \frac{\Delta}{32}|C|,\] we can assume that the
complex numbers $S_\chi(A,B,c\cdot D)$ with $c\in C_2$ all have arguments within
$\frac{1}{2}$ of each other, so that by Lemma \ref{arg} we have
\begin{equation}\label{lowerBound}
\frac{|C_3|}{4|C|}|H_\chi(A,B,C,D)|\leq \left|\sum_{c\in
C_3}S_\chi(A,B,c\cdot D)\right|=|H_\chi(A,B,C_3,D)|
\end{equation}
whenever $C_3$ is a subset of $C_2$.
In particular, if $C_3=C_2$ we have
\[\frac{\Delta^2}{128}|A||B||C||D|\leq\frac{|C_2|}{4|C|}|H_\chi(A,B,C,D)|\leq 
\sum_{d\in D}|S_\chi(A,B,d\cdot C_2)|.\] Now in view of Lemma \ref{energy}, we
see that \begin{align*}\frac{\Delta^2}{128}|A||B||C||D|&\leq|D|\max_{d\in
D}\sqrt{p|A|E_+(B,d\cdot C_2)}\\&\leq
\sqrt{p}|D||A|^{1/2}|B|^{3/4}E_+(C_2,C_2)^{1/4},\end{align*} having bounded $E_+(B,B)$ trivially by $|B|^3$. Thus
\[E_+(C_2,C_2)\geq \frac{\Delta^8}{128^4}|A|^2|B||C|^4p^{-2}\geq
\lr{\frac{\Delta^8}{128^4}|A|^2|B||C|p^{-2}}|C_2|^3.
\] For convenience, write $K^{-1}=\frac{\Delta^8}{128^4}|A|^2|B||C|p^{-2}$. By
Theorem \ref{BSG} there is a subset $C_3\subset C_2$ of size at least
$\frac{|C_2|}{K(\log p)^2}$ and such that \[|C_3-C_3|\ll
K^4\frac{|C_3|^2(\log p)^8}{|C_2|^2}|C_3|.\] In particular, by Theorem
\ref{EnergyEstimate} we have \begin{align*}
E_\times(C_3,C_3)&\ll
|C_3|K^7\lr{\frac{|C_3|^2(\log p)^8}{|C_2|^2}}^{7/4}|C_3|^{7/4}\log
p\\&=K^7|C_3|^{25/4}|C_2|^{-7/2}(\log p)^{15}.
\end{align*}
Inserting this into equation (\ref{lowerBound}), we get
\begin{align*}\frac{\Delta}{4}|A||B||C_3||D|&=\frac{|C_3|}{4|C|}|H_\chi(A,B,C,D)|\\&\leq
|H_\chi(A,B,C_3,D)|\\&\leq\sum_{a\in A}|M_\chi(a+B,C_3,D)|.\end{align*} Next we
apply Lemma \ref{BasicBurgess} to obtain that
\begin{multline*}\frac{\Delta}{4}|A||B||C_3||D|\ll
|A|(|D||C_3|)^{1-\frac{1}{k}}(E_\times(D,D)E_\times(C_3,C_3))^{1/4k}\times\\\times\lr{|B|^{2k}2k\sqrt
p+(2k|B|)^kp}^{1/2k}\end{multline*} which implies (after bounding $E_\times(D,D)$
trivially by $|D|^3$) \[\Delta^{4k} \ll|D|^{-1}|C_3|^{-4}
E_\times(C_3,C_3)\lr{2k\sqrt p+(2k|B|^{-1})^kp}^2.\]
Since $2\leq k\ll \log p$ and $|B|\geq p^\delta$, the final factor is at
most $O(p(\log p)^{2k})$ as long as $k>\frac{1}{2\delta}$, and after inserting
the upper bound for $E_\times(C_3,C_3)$ we have \[\Delta
^{4k}\ll|D|^{-1}
K^7|C_3|^{9/4}|C_2|^{-7/2}(\log
p)^{2k+15}p.\]
Now we substitute $K^{-1}=\frac{\Delta^8}{128^4}|A|^2|B||C|p^{-2}$ and see 
\[\Delta^{4k+56}\ll|D|^{-1}|A|^{-14}|B|^{-7}|C|^{-7}|C_3|^{9/4}|C_2|^{-7/2}(\log
p)^{2k+15}p^{15}.\]
Bounding $|C_3|\leq |C_2|$ and $|C_2|\gg\Delta|C|$ we get
\[\Delta^{4k+\frac{229}{4}}\ll|D|^{-1}|A|^{-14}|B|^{-7}|C|^{-\frac{33}{4}}(\log
p)^{2k+15}p^{15}.\] Upon taking $4k$'th roots we have
\[\Delta^{1+229/16k}\ll
\lr{|D|^{-1}|A|^{-14}|B|^{-7}|C|^{-\frac{33}{4}}p^{15}}^{1/4k}(\log
p)^{1/2+15/4k}.\] Since
\[|D|^4|A|^{56}|B|^{28}|C|^{33}\geq p^{60+\eps},\] the quantity in brackets on
the right is at most $p^{-\eps/4}$. This shows that we must have
$\Delta<p^{-\tau}$ for some $\tau >0$ depending only on $\eps$ and $\delta$.
This is because we only needed $k$ to be sufficiently large in terms of
$\delta$.

If $|C|>\sqrt p$ then we
can break $C$ into a disjoint union of $m\approx \frac{|C|}{\sqrt p}$ sets
$C_1,\ldots, C_m$ of size at most $\sqrt p$.
Then \[|H_\chi(A,B,C,D)|\leq\sum_{j}|H_\chi(A,B,C_j,D)|.\] We obtain a
savings of $p^{-\tau}$ for each $H_\chi(A,B,C_j,D)$ and hence for
$H_\chi(A,B,C,D)$ provided
\[|D|^4|A|^{56}|B|^{28}|C_j|^{33}\gg|D|^4|A|^{56}|B|^{28}p^{33/2}\geq
p^{60+\eps}\] which is guaranteed by hypothesis (with $2\eps$ in place of
$\eps$).
\end{proof}

\end{document}